\documentclass{amsart}
\usepackage{amssymb}
\usepackage{amsmath}
\usepackage{textcomp}
\makeindex

\newcommand{\ba}{\begin{array}}
\newcommand{\eea}{\end{eqnarray}}
\newcommand{\ea}{\end{array}}

\newtheorem{definition}{Definition}[section]
\newtheorem{theorem}[definition]{Theorem}
\newtheorem{lemma}[definition]{Lemma}
\newtheorem{proposition}[definition]{Proposition}
\newtheorem{corollary}[definition]{Corollary}

\newtheorem{remark}[definition]{Remark}

\usepackage[matrix,arrow,curve]{xy}
\usepackage[utf8x]{inputenc}
\usepackage{tikz}

\begin{document}
\title[Thurston's h-principle and Flexibility of Poisson Structures]{Thurston's h-principle and Flexibility of Poisson Structures}
\author[S. Mukherjee]{Sauvik Mukherjee}
\address{Presidency University.\\
e-mail:mukherjeesauvik@yahoo.com\\}

\begin{abstract} 
We prove an analogue of Thurston's h-principle for $2$-dimensional foliations on manifolds of dimension bigger or equal to $4$, in the presence of a fiber-wise non-degenerate $2$-form. This helps us understand the flexibility of rank $2$ regular Poisson structures on open manifolds with dimension bigger or equal to $4$ and it also helps us understand the flexibility of Poisson structures (not regular) on closed $4$-manifolds. 
\end{abstract}
\maketitle

\section{introduction} An h-principle for Poisson structures on open manifolds has been proved by Fernandes and Frejlich in \cite{Fernandes}. We state their result below.\\

 Let $M^{2n+q}$ be a $C^{\infty}$-manifold equipped with a co-dimension-$q$ foliation $\mathcal{F}_0$ and a $2$-form $\omega_0$ such that $(\omega_0^n)_{\mid T\mathcal{F}_0}\neq 0$. Denote by $Fol_q(M)$ and $Dist_q(M)$ the spaces of co-dimension-$q$ foliations and distributions on $M$ respectively identified as a subspace (the entire space in case of distributions) of $\Gamma(Gr_{2n}(M))$, where $Gr_{2n}(M)\stackrel{pr}{\to} M$ be the grassmann bundle, i.e, $pr^{-1}(x)=Gr_{2n}(T_xM)$ and $\Gamma(Gr_{2n}(M))$ is the space of sections of $Gr_{2n}(M)\stackrel{pr}{\to} M$ with compact open topology. Define \[\Delta_q(M),\ \bar{\Delta}_q(M)\subset Dist_q(M)\times \Omega^2(M)\] \[\Delta_q(M):=\{(\mathcal{F},\omega):\omega^n_{\mid T\mathcal{F}}\}\neq 0\] \[\bar{\Delta}_q(M):=\{(\mathcal{D},\omega):\omega^n_{\mid \mathcal{D}}\}\neq 0\] Obviously $(\mathcal{F}_0,\omega_0)\in \Delta_q(M)$. Inthis setting Fernandes and Frejlich has proved the following 

\begin{theorem}(\cite{Fernandes})
\label{Fernandes}
Let $M^{2n+q}$ be an open manifold with $(\mathcal{F}_0,\omega_0)\in \Delta_q(M)$ be given. Then there exists a homotopy $(\mathcal{F}_t,\omega_t)\in \Delta_q(M)$ such that $\omega_1$ is $d_{\mathcal{F}_1}$-closed (actually exact). 
\end{theorem}

In the language of poisson geometry the above result \ref{Fernandes} takes the following form. Let $\pi \in \Gamma(\wedge ^2TM)$ be a bi-vectorfield on $M$, define $\#\pi:T^*M\to TM$ as $\#\pi(\eta)=\pi(\eta,-)$. If $Im(\#\pi)$ is a regular distribution then $\pi$ is called a regular bi-vectorfield.

 \begin{theorem}
 \label{Fernandes-1}
 Let $M^{2n+q}$ be an open manifold with a regular bi-vectorfield $\pi_0$ on it such that $Im(\#\pi_0)$ is an integrable distribution then $\pi_0$ can be homotoped through such bi-vectorfields to a poisson bi-vectorfield $\pi_1$.
 \end{theorem}
 
 \begin{remark}
 Fernandes and Frejlich has also shown by example in \cite{Fernandes} that if the condition that $Im(\#\pi_0)$ being integrable is removed then the result fails. This is because in general a distribution need not have a foliation in its homotopy class. On open manifolds the obstruction is known by Haefliger in \cite{Haefliger}.  
 \end{remark}

In \ref{Fernandes} above $d_{\mathcal{F}}$ is the tangential exterior derivative, i.e, for $\eta \in \Gamma(\wedge^k T^*\mathcal{F})$, $d_{\mathcal{F}}\eta$ is defined by the following formula \[d_{\mathcal{F}}\eta(X_0,X_1,...,X_k)=\Sigma_i(-1)^iX_i(\eta(X_0,..,\hat{X}_i,..,X_k))\]\[+\Sigma_{i<j}(-1)^{i+j}\eta([X_i,X_j],X_0,..,\hat{X}_i,..,\hat{X}_j,..,X_k)\]where $X_i\in \Gamma(T\mathcal{F})$. So if we extend a $\mathcal{F}$-leafwise closed $k$-form $\eta$, i.e, $d_{\mathcal{F}}\eta=0$, to a form $\eta'$ by the requirement that $ker(\eta')=\nu \mathcal{F}$, where $\nu \mathcal{F}$ is the normal bundle to $\mathcal{F}$, then $d\eta'=0$.\\ 

In order to fix the foliation in \ref{Fernandes} one needs to impose an openness condition on the foliation, we refer the readers to \cite{Bertelson} for precise definition of this openness condition. Under this hypothesis Bertelson proved the following

\begin{theorem}(\cite{Bertelson})
If $(M,\mathcal{F})$ be an open foliated manifold with $\mathcal{F}$ satisfies some openness condition and let $\omega_0$ be a $\mathcal{F}$-leaf wise $2$-form then $\omega_0$ con be homotoped through $\mathcal{F}$-leaf wise $2$-forms to a $\mathcal{F}$-leaf wise symplectic form. 
\end{theorem}

She also constructed counter examples in \cite{Bertelson1} that without this openness condition the above theorem fails. A contact analogue of Bertelson's result on any manifold (open or closed) has recently been proved in \cite{Over-twisted} by Borman, Eliashberg and Murphy. \\

 \ref{Fernandes} and \ref{Fernandes-1} has been generalized to closed manifolds in \cite{Mukherjee}. 
 
 \begin{theorem}
 \label{Mukherjee}
 Let $M^{2n+q}$ be a closed manifold with $q=2$ and $(\mathcal{F}_0,\omega_0)\in \Delta_q(M)$ be given. Then there exists a homotopy $\mathcal{F}_t$ of singular foliations on $M$ with singular locus $\Sigma_t$ and a homotopy of two forms $\omega_t$ such that the restriction of $(\omega_t)$ to $T\mathcal{F}_t$ is non-degenerate and $\omega_1$ is $d_{\mathcal{F}_1}$-closed.
 \end{theorem}

In terms of poisson geometry \ref{Mukherjee} states 

\begin{theorem}
\label{Mukherjee-1}
Let $M^{2n+q}$ be a closed manifold with $q=2$ and $\pi_0$ be a regular bi-vectorfield of rank $2n$ on it such $Im(\#\pi_0)$ is integrable distribution. Then there exists a homotopy of bi-vectorfields $\pi_t,\ t\in I$ (not regular) such that $Im(\#\pi_t)$ integrable and $\pi_1$ is a poisson bi-vectorfield.
\end{theorem}

Now we come to the main topic of this paper. In \cite{Thurston} Thurston has proved the following 

\begin{theorem}
\label{Thurston}
Any $C^{\infty}$ $2$-plane field on a manifold $M$ of dimension at least $4$ is homotopic to an integrable one. Relative version of this result is also true.
\end{theorem}

In this paper we prove an analogue of this result in the presence of a fiber-wise non-degenerate $2$-form. The main theorem of this paper is the following. For a $2$-distribution $\mathcal{D}$ and for a section $\tau$ of $\mathcal{D}$, we set $\omega_{\mid \mathcal{D}}=\omega_{\mid \tau}$.

\begin{theorem}
\label{Main}
Let $M^n$ be a manifold of dimension atleast $4$, i.e, $n\geq 4$ and $(\tau_0,\omega_0)\in \bar{\Delta}_{n-2}(M)$. Then $(\tau_0,\omega_0)$ is homotopic to $(\tau_1,\omega_1)$ through a homotopy $(\tau_t,\omega_t)\in \bar{\Delta}_{n-2}(M)$, $t\in [0,1]$ such that $\tau_1$ is integrable.
\end{theorem}

Using \ref{Main} in \ref{Fernandes} and \ref{Mukherjee} (or equivalently in \ref{Fernandes-1} and \ref{Mukherjee-1}) we get the following.

\begin{corollary}
\label{Main-coro}
If there exists a regular rank $2$ bivector field $\pi_0$ on an open manifold $M$ of dimension bigger or equal to $4$ then there exists a homotopy $\pi_t,\ t\in [0,1]$ of regular bivector fields of rank $2$ such that $\pi_1$ is regular Poisson.\\

If on the other hand $M$ is closed $4$-manifold and $\pi_0$ a regular rank $2$ bivector field on it then there would exists a homotopy of bivector fields $\pi_t,\ t\in [0,1]$ (not regular) such that $\pi_1$ is Poisson (not regular). 
\end{corollary} 

\section{Plan of the proof of \ref{Main}} We closely follow \cite{Mitsumatsu} in order to prove \ref{Main}. We state the following proposition.

\begin{proposition}
\label{Key}
Let $(\tau_0,\omega_0)\in \bar{\Delta}_{n-2}(\mathbb{R}^{n})$ then $(\tau_0,\omega_0)$ is homotopic through $(\tau_t,\omega_t)\in \bar{\Delta}_{n-2}(\mathbb{R}^{n}),\ t\in [0,1]$ such that $\tau_1$ is integrable in a neighborhood of $[-1,1]^n$ and $(\tau_t,\omega_t)$ is constant in a neighborhood of $\mathbb{R}^n-(-2,2)^n$.
\end{proposition}

\begin{proof}
{\bf(Proof of \ref{Main})} Let $\{h_i:U_i\to \mathbb{R}^n\}_{i\in \mathbb{N}}$ be a countable atlas with the property that $\{h_i^{-1}((-1,1)^n)\}_{i\in \mathbb{N}}$ covers $M$. Define \[(\tau'_0,\omega'_0)=((h_1)_*(\tau_0)_{\mid U},(h_1)_*\omega_0)\] Using \ref{Key} for $(\tau'_0,\omega'_0)$ we get $(\tau'_t,\omega'_t)\in \bar{\Delta}_{n-2}(\mathbb{R}^n)$ such that $\tau'_1$ is integrable in a neighborhood of $[-1,1]^n$ and $(\tau'_t,\omega'_t)$ is constant on a neighborhood of $\mathbb{R}^n-(-2,2)^n$.\\

$h_1^*\tau'_1$ can be extended by $\tau$ to a plane field $\tau_1$ and do the same for $h_1^*\omega'_1$ and let us call it $\omega_1$.\\

Continue by replacing $\tau_0,\ \omega_0\ and\ h_1$ by $\tau_1,\omega_1\ and\ h_2$. If $x\in \cup_{i=1}^N h_i^{-1}((-1,1)^n)$, then there is a neighborhood $V$ of $x$ such that all homotopies after $N$ steps are constant and if $V$ is small enough, all homotopies from step $1$ to $N$ are constant on $V$ and hence the sequence $(\tau_0,\omega_0), (\tau_1,\omega_1),...$ converges.  
\end{proof}

So the proof of \ref{Key} is the main task. It shall be done in the following three steps namely 

\begin{enumerate}
\item Triangulate $\mathbb{R}^n$ so that $\tau_0$ is in general position on a neighborhood of $[-1,1]^n$.\\
\item Deform $(\tau_0,\omega_0)$ into one which is civilized in near $[-1,1]^n$ with respect to the triangulation in the first step.\\
\item Filling the holes.\\
\end{enumerate}
 
 These steps will be covered in the following sections bellow.\\
 
 \section{Triangulation in general position} This part is exactly same as \cite{Mitsumatsu}, we just outline the main definitions and results for completion.
 
 \begin{definition}
 \label{General Position}
 For a $k$-plane field $\tau$ on $\mathbb{R}^n$, an $n$-simplex $\sigma$ is in general position with respect to $\tau$ if for all $x\in \sigma$ the orthogonal projection along $\tau(x)$ from the tangent plane of every $(n-k)$-face to $\tau(x)^{\perp}$ is injective. A triangulation is said to be in general position with respect to $\tau$ in a neighborhood of a closed set if every $n$-simplex of this triangulation intersecting the closed set is in general position.
 \end{definition}
 
 \begin{definition}
 \label{jiggling}
 An affine triangulation $T'$ is said to be an $\varepsilon$-jiggling of a given affine triangulation $T$ for a given $\varepsilon>0$ if there is a simplicial isomorphism $T \stackrel{\phi}\to T'$ such that \[|\phi(v)-v|<\varepsilon\ for\ every\ vertex\ v\in T\]
 \end{definition}
 
 \begin{proposition}(\cite{Thurston})
 \label{Thurston} 
 For a given $C^1$ $k$-plane field $\tau$ on $\mathbb{R}^n$ with $1\leq k \leq n-1$, a compact set $K\subset \mathbb{R}^n$ and $\varepsilon>0$ there exists an $L\in \mathbb{N}$ such that whenever $l \geq L$ there exists an $\varepsilon$-jiggling of the triangulation associated to the cubical lattice $(\frac{1}{l}\mathbb{Z})^n$ which is in general position with respective to $\tau$ near $K$.
 \end{proposition}

\section{Civilization} Let $T$ be an $\varepsilon$-jiggling of the standard triangulation associated to the lattice $(\frac{1}{l}\mathbb{Z})^n$ which is in general position with respect to $\tau$ in a neighborhood of $[-2,2]^n$. Take $l$ large enough so that the following holds. \\

\begin{enumerate}
\item[{\bf(A)}] If $x$ is a vertex of $T$ with $\bar{st}(x,T)\cap [-1,1]^n$ being non-empty then $\bar{St}(x,T)\subset [-\frac{3}{2},\frac{3}{2}]^n$.\\
\item[{\bf(B)}] If $\sigma$-is a $n$-simplex of $T$ which intersects $[-2,2]^n$ then for any $x,y\in \sigma$ the plane field $\tau(y)$ is the graph of a linear map $L_{xy}:\tau(x)\to \tau(x)^{\perp}$ of norm less that one.
\end{enumerate}

Let $T_1$ be the union of all simplices which are faces of $n$-simplices of $T$ intersecting $[-1,1]^n$.\\

We shall deform $(\tau,\omega)$ into one whose plane field is integrable in a neighborhood of the $(n-1)$-skeleton of $T_1$. The deformation will be done through $(\tau_t,\omega_t)\in \bar{\Delta}_{n-2}(\mathbb{R}^n)$ such that $T$ is in general position with respect to $\tau_t$ near $[-2,2]^n$ and which satisfies {\bf(B)}.\\

We first give the definition of civilized pair $(\tau,\omega)$.

\begin{definition}
\label{Civilization-Def}
Let $0\leq j \leq n-2$. We say $(\tau_j,\omega_j)$ civilized on the $j$-skeleton of $T_1$ if $T_1$ is in general position with respect to $\tau_j$ near $[-2,2]^n$ and there are real numbers $\delta_0>...>\delta_j>0\ and\ \eta_0>...>\eta_j>0$ satisfying
 
 \begin{enumerate}
 \item[{\bf(C)}] Let $\sigma$ be an $i$-simplex of $T_1$ with $0\leq i\leq j$. For $x\in \sigma$ consider the planes $x+\tau_j(x)$ and $x+E_x$, where $E_x=(\tau_j(x)+T\sigma)^{\perp}$. Let $B_x(\delta)$ and $E_x(\eta)$ be the closed $\delta$ and $\eta$ neighborhoods of $x$ in the planes $x+\tau_j(x)$ and $x+E_x$ respectively. From \ref{General Position} it follows that the disk $B_x(\delta)\times E_x(\eta)$ has dimension $(n-i)$. Then the disk $B_x(\delta_i)\times E_x(\eta_i)$ is the fiber of a tubular neighborhood $N(\sigma)$ of $\sigma$ in $\mathbb{R}^n$ moreover any $(n-2)$-simplex of $T_1$ which has $\sigma$ as a face intersects the boundary of the disk $B_x(\delta_i)\times E_x(\eta_i)$ in a subset of $int(B_x(\delta_i))\times \partial(E_x(\eta_i))$.\\
 
 \item[{\bf(D)}] On the fibers $B_x(\delta_i)\times E_x(\eta_i)$, $(\tau_j,\omega_j)=(\tau_j(x),\omega_j(x))$.\\
 
 \item[{\bf(E)}] For two simplices $\sigma,\sigma'$ of $T_1$ of dimension less or equal to $j$, we must have $N(\sigma)\cap N(\sigma')\subset N(\sigma \cap \sigma')$. If $\sigma'$ is a proper face of $\sigma$ and $N(\sigma)\cap N(\sigma')$ non-empty so that $(y_1,y_2)\in B_x(\delta_i)\times E_x(\eta_i)$ lies in $B_v(\delta_{i'})\times \{y'_2\}$ with $v\in \sigma'$ and $y'_2\in E_v(\eta_{i'})$ with $i'=dim\sigma'$, then we must have $B_x(\delta_i)\times \{y_2\}\subset int(B_v(\delta_{i'}))\times \{y'_2\}$. We must also have $N(\sigma)\cap \sigma''$ be empty, where $\sigma''$ is of dimension at least $j+1$ for which $\sigma$ is not a face.\\
 
 \end{enumerate}
 
 For the case $j=n-1$ the conditions {\bf(C)}, {\bf(D)} and {\bf(E)} will have to be replaced by {\bf(C')}, {\bf(D')} and {\bf(E')} respectively.
 
 \begin{enumerate}
 \item[{\bf(C')}] For a $(n-1)$-simplex $\sigma$ of $T_1$ let $F_x(\delta)$ be the closed $\delta$-neighborhood of $x$ in the line $x+F_x$, where $F_x$ is the the orthogonal complement of $\tau_{n-1}(x)\cap T\sigma$ in $\tau_{n-1}(x)$. Then $F_x(\delta_{n-1})$ is the fiber of a tubular neighborhood $N(\sigma)$ of $\sigma$ in $\mathbb{R}^n$.\\
 
 \item[{\bf(D')}] If $x\in \sigma$, where $\sigma$ is a $(n-1)$-simplex of $T_1$. Then on $F_x(\delta_{n-1})$ the pair $(\tau_{n-1},\omega_{n-1})$ is equal to $(\tau_{n-1}(x),\omega_{n-1}(x))$.\\
 
 \item[{\bf(E')}] {\bf(E)} holds and if $\sigma,\sigma'$ be $(n-1)$-simplices then we must have $N(\sigma)\cap N(\sigma')\subset N(\sigma\cap\sigma')$. If $\sigma''$ is a proper face of $\sigma$ so that $F_x(\delta_{n-1})$ intersects $N(\sigma')$, say $y\in F_x(\delta_{n-1})$ also belongs to $B_v(\delta_{i'})\times \{y'_2\}$ with $v\in \sigma'$ and $y'\in E_v(\eta_{i'})$, $i'=dim\sigma'$, then $F_x(\delta_{n-1})\subset int(B_v(\delta_{i'}))\times \{y'\}$. Moreover $N(\sigma)\cap \sigma'''$ is empty for any $n$-simplex $\sigma'''$ for which $\sigma$ is not a face.
 \end{enumerate}
\end{definition}

Let us set that for any $2$-plane field for which $T$ is in general position near $[-2,2]^n$ to be civilized on the $-1$-skeleton of $T_1$ and $\delta_{-1},\eta_{-1}=\infty$.

\begin{proposition}
\label{Civilization-Prop}
Let $-1\leq p\leq n-2$ and let $(\tau_{p-1},\omega_{p-1})\in \bar{\Delta}_{n-2}(\mathbb{R}^n)$ which is civilized on the $(p-1)$-skeleton of $T_1$ and let $\delta_0>...>\delta_{p-1}>0$ and $\eta_0>...>\eta_{p-1}>0$ be the associated real numbers. Then there exists \[(\tau_p,\omega_p)\in \bar{\Delta}_{n-2}(\mathbb{R}^n),\ 0<\delta_p<\delta_{p-1}\ and\ 0<\eta_p<\eta_{p-1}\] such that properties {\bf(C)} to {\bf(E)} holds for $j=p$ (for $j=p=n-1$ {\bf(C)} to {\bf(E)} has to be replaced by {\bf(C')} to {\bf(E')}) and $\delta_i,\ \eta_i,\ i=0,...,p$. Moreover $(\tau_p,\omega_p)$ is homotopic to $(\tau_{p-1},\omega_{p-1})$ in $\bar{\Delta}_{n-2}(\mathbb{R}^n)$ through a homotopy for which $T$ remains in general position near $[-2,2]^n$ and which satisfies {\bf(B)} through out the homotopy.
\end{proposition}

\begin{proof} We shall only deal with the case $0\leq j\leq n-2$. The $j=n-1$ case is similar.\\

Consider all $p$-simplices $\sigma$ of $T_1$. Let $\sigma'$ be a proper face of $\sigma$. Then $(\tau_{p-1},\omega_{p-1})$ is constant on $B_y(\delta_i)\times E_y(\eta_i)$, $i=dim\sigma'$ and $\sigma$ intersects $B_y(\delta_i)\times E_y(\eta_i)$ in $int(B_y(\delta_i))\times E_y(\eta_i)$ by {\bf(C)}. So if $x\in \sigma \cap B_y(\delta_i)\times E_y(\eta_i)$ then $\tau_{p-1}(x)=\tau_{p-1}(y)$ and $E_x\leq E_y$ where $E_x=(T\sigma+\tau_{p-1}(x))^{\perp}$. So we can find $\delta_p$ and $\eta_p$ so that {\bf(C)} and {\bf(E)} holds and {\bf(D)} holds for those \[x\in \sigma \cap N(\sigma')\] where $\sigma'$ a proper face of $\sigma$. The last condition of {\bf(C)} holds if $\frac{\eta_p}{\delta_p}$ is small.\\

So we need to deform $(\tau_{p-1},\omega_{p-1})$ so that {\bf(D)} holds for all $x\in \sigma$ and the deformation will be supported in the complement of $N(\sigma')$. \\

By {\bf(B)} and {\bf(E)}, for all $x\in \sigma$ and $z\in B_x(\delta_p)\times E_x(\eta_p)$ the plane $\tau_{p-1}(z)$ is the graph of a linear map $L_{zx}:\tau_{p-1}(x)\to \tau_{p-1}(x)^{\perp}$.\\

Choose $\bar{\delta}_p>\delta_p\ and\ \bar{\eta}_p>\eta_p$ such that $B_x(\bar{\delta}_p)\times E_x(\bar{\eta}_p)$ are still fibers of a tubular neighborhood $\bar{N}(\delta)$.\\

Now on $B_x(\delta_p)\times E_x(\eta_p)$ we set \[(\tau_p,\omega_p)=(\tau_{p-1}(x),\omega_{p-1}(x))\] Let the discs $B_x(\delta_p)\times E_x(\eta_p)$ and $B_x(\bar{\delta}_p)\times E_x(\bar{\eta}_p)$ have radii $r,\bar{r}$ respectively. Define a continuous map $f:[r,\bar{r}]\to [0,\bar{r}]$ such that $f(r)=0$ and $f(\bar{r})=\bar{r}$. Now for $y\in B_x(\bar{\delta}_p)\times E_x(\bar{\eta}_p)\cap B_x(\delta_p)\times E_x(\eta_p)$ set $(\tau_p(y),\omega_p(y))=(\tau_{p-1}(f(|y|)y),\omega_{p-1}(f(|y|)y))$.
\end{proof}

\section{Filling The Hole} Now we can assume that $(\tau_0,\omega_0)$ satisfies the desired properties of \ref{Key} on $N(\partial \sigma)$, where $\sigma$ is an $n$-simplex of $T_1$. A subset of $\sigma$ diffeomorphic to $B^2\times B^{n-2}$ containing the complement of $N(\partial \sigma)$ is called a hole.

\begin{proposition}
\label{Filling-1}
For each $n$-simplex $\sigma$ of $T_1$ there exists an embedding $\phi:D^2\times D^{n-2}\to int(\sigma)$ such that $(\tau',\omega')=(\phi^*\tau,\phi^*\omega)\in \bar{\Delta}_{n-2}(D^n)$ satisfying

\begin{enumerate}
\item Near $D^2\times \partial D^{n-2}$, $\tau'$ is the kernel of the projection to $D^{n-2}$ factor.\\
\item $\tau'$ is $\pitchfork$ to $\partial D^2\times D^{n-2}$ and in a neighborhood of $\partial D^2\times D^{n-2}$ is the pull back of the line field $\tau' \cap T(\partial D^2\times D^{n-2})$ by the projection $(D^2-\{0\})\times D^{n-2}\to \partial D^2\times D^{n-2}$. Furthermore the line field $\tau' \cap T(\partial D^2\times D^{n-2})$ is $\pitchfork$ to $\{x\}\times D^{n-2}$ for all $x\in D^2$.
\end{enumerate}
\end{proposition}

\begin{proof}
The proof is same as (4.1)-(4.3) in \cite{Mitsumatsu}. Only one needs to observe that the non-degeneracy condition is preserved by pull back under diffeomorphisms.
\end{proof}

\subsection{Another Pair} For a $k$-manifold $M$ a foliated $M$-product over the circle $S^1$ is a co-dimension $k$ foliation on $S^1\times M$ which is $\pitchfork$ to the second factor. A foliated $M$-product over $S^1$ is said to have compact support if there exists a compact set $C\subset M$ such that on $S^1\times (M-C)$ the foliation is given by the projection onto $M-C$.\\

Let $\xi$ be a foliated $M$-product with compact support and let $\alpha \in \Omega^1(S^1\times M)$. Consider \[(\xi,\alpha)\in Fol_k(S^1\times M)\times \Omega^1(S^1\times M)\] satisfying $\alpha_{\mid T\xi}\neq 0$.  By \ref{Filling-1} $\tau'$ defines a foliated $\mathbb{R}^{n-2}$-product on $S^1\times \mathbb{R}^{n-2}$. Here we identify $int(D)^{n-2}$-product over $S^1$ with $\mathbb{R}^{n-2}$-product over $S^1$.\\

Define the vector field $X:=\frac{\tau'_{\mid S^1\times D^{n-2}}}{T\xi'}$, where $\xi'$ is the foliated $\mathbb{R}^{n-2}$ structure as identified with $\tau'$ by \ref{Filling-1}. Then define $\alpha'=\omega'_{\mid S^1\times D^{n-2}}(X,-)$. Observe that the pair $(\xi',\alpha')$ satisfies the non-degeneracy condition $\alpha'_{\mid T\xi'}\neq 0$. From now on whenever we write the pair $(\xi,\alpha)$ we mean that it satisfies this non-degeneracy condition.\\

Thus we need to find a pair $(\tau'',\omega'')\in \bar{\Delta}_{n-2}(D^2\times \mathbb{R}^{n-2})$ with $\tau''=T\mathcal{F}''$ integrable satisfying 

\begin{enumerate}
\item[(*)]  Outside some compact set $D^2\times C$, $\tau''$ is given by the projection onto $\mathbb{R}^{n-2}$.\\
\item[(**)] $\tau''\pitchfork rS^1\times \mathbb{R}^{n-2}$ for $r$ close to $1$ and induces there the given pair $(\xi',\alpha')$.\\
\item[(***)] $(\tau'',\omega'')$ is homotopic in $\bar{\Delta}_{n-2}(S^1\times \mathbb{R}^{n-2})$ to an one for which the $2$-plane field is $\pitchfork$ to the $\mathbb{R}^{n-2}$-factor by a homotopy constant near $S^1\times \mathbb{R}^{n-2}$. \\
\end{enumerate}

Now as in \cite{Mitsumatsu} we identify $\xi'$ with a periodic path $\gamma_{\xi'}:\mathbb{R}\to Diff_c\mathbb{R}^{n-2}$. 

\begin{definition}
\label{Fillability}
We call a pair $(\xi,\alpha)$ or equivalently $(\gamma_{\xi},\alpha)$ fillable if there there exists a pair $(\tau=T\mathcal{F},\omega)$ inducing $(\xi,\alpha)$ on $S^1\times \mathbb{R}^{n-2}$ and which satisfies (*) and (**). 
\end{definition}

\begin{definition}
\label{Horizontal-Adjusted}
A path $\gamma:[0,1]\to Diff_c\mathbb{R}^{n-2}$ starting at identity $id$ is called horizontal on an interval $J\subset [0,1]$ if it is constant on $J$. A path which is horizontal near the starting and end is called adjusted.
\end{definition}

{\bf Concatenation:} Let $(\gamma_i,\alpha_i),\ i=1,2$ be two adjusted pairs. Define the immersion \[\{(exp(2\pi it),\gamma_1(2t,y)):y\in \mathbb{R}^{n-2},\ t\in [0,1/2]\}\stackrel{\phi_1}{\to}\{(exp(2\pi it),\gamma_1(t,y)):y\in \mathbb{R}^{n-2}, t\in[0,1]\}\] by $\phi_1(exp(2\pi it),\gamma_1(2t,y))=(exp(2\pi i.2t),\gamma_1(2t,y))$. Set $\bar{\alpha}_1=\phi^*\alpha_1$.\\

 Similarly define \[\{(exp(2\pi it),\gamma_2(2t-1,\gamma_1(1,y))):y\in \mathbb{R}^{n-2},\ t\in [1/2,1]\}\] \[\stackrel{\phi_2}{\to}\{(exp(2\pi it),\gamma_2(t,\gamma_1(1,y))):y\in \mathbb{R}^{n-2}, t\in[0,1]\}\] by $\phi_2(exp(2\pi it),\gamma_2(2t-1,\gamma_1(1,y)))=(exp(2\pi i(2t-1)),\gamma_2(2t-1,\gamma_1(1,y)))$. Set $\bar{\alpha}_2=\phi^*\alpha_2$. \\
 
 The new pairs still satisfies the non-degeneracy condition on the respective domains. Now define the concatenation $(\gamma_1,\alpha_1)*(\gamma_2,\alpha_2)=(\gamma,\alpha)$ as follows. The concatenation of the paths namely $\gamma$ is given by $\gamma(t)=\gamma_1(2t),\ t\in [0,1/2]$ and $\gamma(t)=\gamma_2(2t-1)\circ \gamma_1(1),\ t\in [1/2,1]$. From the fact that $\gamma_i$'s being adjusted it follows that $\bar{\alpha}_i$'s match nicely and we get $\alpha$.\\

\begin{lemma}
\label{Concatenation}
The concatenation of two fillable adjusted pair $(\gamma_1,\alpha_1)$ and $(\gamma_2,\alpha_2)$ is fillable.
\end{lemma}

\begin{proof}
Let us consider $(\gamma_1,\alpha_1)$. Define \[\{(r exp(2\pi it),\gamma_1(2t,y)):y\in \mathbb{R}^{n-2},\ t\in [0,1/2],\ r\in [0,1]\}\] \[\stackrel{\Phi_1}{\to}\{(r exp(2\pi it),\gamma_1(t,y)):y\in \mathbb{R}^{n-2}, t\in[0,1],\ r\in [0,1]\}\] by $\Phi_1(r exp(2\pi it),\gamma_1(2t,y))=(r exp(2\pi i.2t),\gamma_1(2t,y))$. Set $(\bar{\mathcal{F}}_1,\bar{\omega}_1)=(\Phi^*\mathcal{F}_1, \Phi^*\omega_1)$, where $(T\mathcal{F}_1,\omega_1)\in \bar{\Delta}_{n-2}(D^2\times \mathbb{R}^{n-2})$ be the pair given by the fillability of $(\gamma_1,\alpha_1)$. Again the new pair satisfies the non-degeneracy condition.\\

Do the similar arrangement in view of the concatenation arrangement. Now the rest of the proof is same as {\bf Lemma 4.10} in \cite{Mitsumatsu} and we skip it.

 \end{proof}
 
 \begin{remark}
 \label{Product-concatenation}
 Observe that in the concatenation process $\gamma(1)=\gamma_2(1)\circ \gamma_1(1)$ and we shall treat the product $\gamma_2(1)\circ \gamma_1(1)$ and concatenation in place of each other without mention.
 \end{remark}

\subsection{An Example}
\label{Example}
 Consider the foliation on the torus $S^1\times S^1$ by lines of constant slope $a$, i.e, the foliation is given by the kernel of the one form $d\theta-ad\phi$ where $(\theta,\phi)$ are coordinates on $\S^1\times S^1$. Let $D^2$ have the polar coordinate $(r,\phi)$ and let $\lambda_0,\lambda_{1/2}\lambda_1$ be cutoff functions on $[0,1]$ such that $\lambda_i=1$ on a neighborhood of $i\in [0,1]$ and $supp(\lambda_0)\cap supp(\lambda_1)$ is empty. Then the kernel of the one form \[\lambda_1(r)(d\theta-ad\phi)+\lambda_{1/2}(r)dr+\lambda_0(r)d\theta\] on $D^2\times S^1$ with coordinates $(r,\phi,\theta)$ defines a co-dimension-$1$ foliation $\mathcal{F}$ on the solid torus $D^2\times S^1$ whose restriction on the boundary torus $S^1\times S^1$ is the given foliation, i.e, $Ker(\theta-ad\phi)$. On the inner solid torus $D_{1/2}^2\times S^1$ the foliation is given by the taut foliation.\\

On the torus consider the one form $\alpha=ad\theta+d\phi$. Observe that \[Ker(d\theta-ad\phi)=\langle(a\partial_{\theta}+\partial_{\phi})\rangle\] So $\alpha_{\mid Ker(d\theta-ad\phi)}\neq 0$.\\

Now consider the two form $\omega$ whose restriction on $\{(r,\phi):0\leq r \leq 1/2\}\times S^1$ is given by $(\lambda_0(r)rdr+\lambda_{1/2}(r)d\theta)\wedge d\phi$ and whose restriction on $\{(r,\phi):1/2\leq r\leq 1\}\times S^1$ is given by \[\frac{1}{2}(2\lambda_1(r)dr+\lambda_{1/2}(r)d\theta)\wedge d\phi+\frac{1}{2}d\theta \wedge (\lambda_{1/2}(r)d\phi-2a\lambda_1(r)dr)\] Observe that on the boundary torus $S^1\times S^1$, $\omega(\partial_r,-)=\alpha$ and $\omega_{\mid T\mathcal{F}}\neq 0$.\\

A tubular neighborhood of $S^1\subset \mathbb{R}^{n-2}$ is given by $S^1\times D^{n-3},\ [n\geq 4]$. Consider a periodic curve $[0,1]\to Diff_c\mathbb{R}^{n-2}$ which is constant outside $S^1\times D^{n-3}$ and whose restriction to $S^1\times D^{n-3}$ induces a foliated $S^1\times D^{n-3}$-product structure $\xi$ over $S^1$, i.e, a co-dimension-$(n-2)$ foliation on $S^1\times S^1\times D^{n-3}$ whose restriction on $S^1\times S^1\times \{x\},\ x\in D^{n-3}$ is given by $Ker(d\theta-f(x)d\phi)$, where $f:D^{n-3}\to [0,1]$ is a non-vanishing map which is zero near $\partial D^{n-3}$.\\

Observe that if we set $f(x)=a$ we get the foliation mentioned above. Let $\omega(x)$ be the two form achieved by replacing $a$ by $f(x)$ in $\omega$. Now one just need to interpolate by introducing the function $g:D^{n-3}\to [0,1]$ such that $g\equiv 0$ in a neighborhood of $\partial D^{n-3}$ and $g\equiv 1$ on $supp(f)$. The resulting one form whose kernel defines the foliation is given by \[(1-g(x))d\theta+g(x)(\lambda_1(r)(d\theta-f(x)d\phi)+\lambda_{1/2}(r)dr+\lambda_0(r)d\theta)\] and the two form is given by \[(1-g(x))rdr\wedge d\theta+g(x)\omega(x)\] 

\subsection{The group $Diff_c\mathbb{R}^{n-2}$} In this subsection we recall some results regarding the group $Diff_c\mathbb{R}^{n-2}$. Let us set $k-n-2$. Let $C^{\infty}_K(\mathbb{R}^k,\mathbb{R}^k)$ be the vector space of smooth maps from $\mathbb{R}^k$ to itself with support a compact set $K\subset \mathbb{R}^k$ and equipped with $C^{\infty}$-topology. The inductive limit of $C^{\infty}_K(\mathbb{R}^k,\mathbb{R}^k)$ on $K$ will be denoted by $C^{\infty}_c(\mathbb{R}^k,\mathbb{R}^k)$. Now $Diff_c\mathbb{R}^{n-2}$ is a manifold modeled on $C^{\infty}_c(\mathbb{R}^k,\mathbb{R}^k)$. An atlas is given by different translates by elements $g\in Diff_c\mathbb{R}^k$ of small enough neighborhoods $U_g$ depending on $g$ such that $U_g$ is so small that $g+U_g\subset Diff_c\mathbb{R}^k$. For any map $f:J\to C^{\infty}_c(\mathbb{R}^k,\mathbb{R}^k)\subset Diff_c\mathbb{R}^k$ is contained in some $C^{\infty}_K(\mathbb{R}^k,\mathbb{R}^k)$ for some compact set $K$. For more detailed information we refer to \cite{Mitsumatsu}.

\begin{proposition}(\cite{Haller})
\label{HaRyTe}
For $k\geq 2$ and $B\subset \mathbb{R}^k$ be open and bounded. Then there exists compactly supported smooth vector fields $X_1,...,X_6$ on $\mathbb{R}^k$, a $C^{\infty}$-open neighborhood $\mathcal{W}$ of the identity in $Diff^{\infty}_c(B)$, and a smooth mappings $\sigma_1,...,\sigma_6:\mathcal{W}\to Diff^{\infty}_c\mathbb{R}^k$ such that for all $g\in \mathcal{W}$, \[g=[\sigma_1(g),expX_1]\circ ...\circ [\sigma_6(g),expX_6]\] where the commutator $[a,b]=aba^{-1}b^{-1}$. Moreover the vector fields are close to zero.
\end{proposition}

\begin{proposition}(\cite{Tsuboi})
\label{Tsu}
Let $U$ be an open subset of a manifold $M$ and let $h$ be a diffeomorphism of $M$ such that $U\cap f(U)$ is empty also assume that $a,b\in Diff^{\infty}_c(M)$ are supported in $U$, then the commutator is the product of four conjugates of $h$ and $h^{-1}$, more precisely \[[a,b]=h(ch^{-1}c^{-1})(bchc^{-1}b^{-1})(bh^{-1}b^{-1})\] where $c=h^{-1}ah$.
\end{proposition}

\begin{corollary}(\cite{Mitsumatsu})
\label{Tsu-cor}
Let $M^k$ be a smooth connected manifold and let $h\in Diff^{\infty}_c(M^k)$ any element other than identity. Let $a_i,b_i,i=1,...,r$ be elements of $Diff^{\infty}_c(M^k)$ such that for each $i$ the diffeomorphisms $a_i,\ b_i$ have support in $int(U_i)$ where $U_i$ is a closed $k$-ball in $M^k$. Then \[f:=\Pi_{i=1}^r [a_i,b_i]\ is\ a\ product\ of\ 4r\ conjugates\ of\ h\ h^{-1}\]
\end{corollary}

\subsection{Concluding The Proof of \ref{Key}} We consider a pair $(\gamma,\alpha)$ and show that it is fillable. The proof is similar to \cite{Mitsumatsu}. \\

\begin{enumerate}
\item[(a)] Define $V_{\varepsilon}:=\{id+e:e\in C^{\infty}_c(\mathbb{R}^k,\mathbb{R}^k)\ with\ max_{x}|de_x|<\varepsilon\}$. $V_1$ is an open contractible neighborhood of $id$ in $Diff^{\infty}_c(\mathbb{R}^{n-2})$. Then as in \cite{Mitsumatsu} there exists $\varepsilon>0$ such that any composition of $72$ elements of $V_{\varepsilon}$ is in $V_1$.\\

\item[(b)] Recall from \ref{Example} that the periodic curve of the pair that has been filled is of the form $h_f(t,z)=z,\ z\notin S^1\times D^{n-3}$ and $h_f(t,z)=(\theta+t.f(x),x),\ z=(\theta,x)\in S^1\times D^{n-3}$. Here $f:D^{n-3}\to [0,1]$ is a smooth map vanishing near the boundary. Observe that $h_f(t)^{-1}=h_{-f}(t)$. So for $f$ small enough $h_f(t)^{-1},h_{-f}(t)\in V_{\varepsilon}$. Set $h=h_f(1)$.\\

\item[(c)] $U\subset S^1\times D^{n-3}$ be an open ball such that $U\cap h(U)$ is empty. Let $B \subset \mathbb{R}^{n-2}$ be an open ball which contains $S^1\times D^{n-3}$ and $A$ be an open ball in $\mathbb{R}^{n-2}$ containing $\bar{B}$.\\

\item[(d)] $g_t:\mathbb{R}^{n-2}\to \mathbb{R}^{n-2},\ t\in [0,1]$ be a compactly supported isotopy with $g_0=id$ and $g_1(\bar{A})\subset U$. Set $g=g_1$.\\

\item[(e)] By \ref{HaRyTe} there exists $\mathcal{W}\subset V_1$ a $C^{\infty}$-open neighborhood of $id$ in $Diff^{\infty}_c(B)$, $\sigma_i:\mathcal{W}\to Diff_c^{\infty}(A),\ i=1,...,6$ and compactly supported vector fields $X_i,\ i=1,...,6$. We can assume that $supp(X_i)\subset A,\ i=1,...,6$. We shall later refer how small $\mathcal{W}$ needs to be.\\   

\end{enumerate}

\begin{remark}
\label{Conjugate-fillable}
Now observe that $(\gamma,\alpha)$ is fillable if and only if $(g\circ \gamma \circ g^{-1},(g^{-1})^*\alpha)$ is fillable.
\end{remark}

 As in \cite{Mitsumatsu} $g\circ \gamma \circ g^{-1}$ is the concatenation of of $g\circ \gamma_i \circ g^{-1},\ i=1,...,q$, where $\gamma_i(t)=\gamma(\frac{i+t}{q})\circ \gamma(\frac{i}{q})^{-1}$. All of then have images in $\mathcal{W}$.\\

By \ref{HaRyTe} each $\gamma_i(1)$ is a product of the commutators $[\sigma_j(\gamma_i(1)),expX_j],\ j=1,...,6$. As we know that the conjugate of the commutator is equal to the commutator of the conjugates, we get \[g[\sigma_j(\gamma_i(1)),expX_j]g^{-1}=h\circ ((h^{-1}g\sigma_j(\gamma_i(1))g^{-1}h)\circ h^{-1}\circ (h^{-1}g\sigma_j(\gamma_i(1))^{-1}g^{-1}h))\circ\] \[ ((g\ expX_jg^{-1})\circ (h^{-1}g\sigma_j(\gamma_i(1))g^{-1}h)\circ h \circ (h^{-1}g\sigma_j(\gamma_i(1))^{-1}g^{-1}h)\circ (g\ expX_j^{-1}g^{-1}))\circ \] \[((g\ expX_jg^{-1})\circ h^{-1} \circ (g\ expX_j^{-1}g^{-1}))\] 

Observe that it is a product of $12$ diffeomorphisms and $j$ runs from $1$ to $6$. So in total it is a product of $72$ diffeomorphisms. So we choose $\mathcal{W}$ from item (e) above accordingly.\\

Also observe that in the above expression for $g[\sigma_j(\gamma_i(1)),expX_j]g^{-1}$ is a product of different conjugates of $h$ and $h^{-1}$. So from \ref{Conjugate-fillable} and \ref{Product-concatenation} it is fillable.

\end{document}